\pdfoutput=1
\documentclass[11pt]{amsart}
\usepackage{geometry}             
\usepackage[parfill]{parskip}    \usepackage{graphicx,color,amsthm}
\usepackage{amssymb}
\usepackage{epstopdf}
\usepackage{pdfsync}
\usepackage{caption}
\usepackage{tikz-cd}
\usepackage{subcaption}
\usepackage{mathtools}
\theoremstyle{plain} 
\newtheorem{thm}{Theorem}

\newtheorem{prop}[thm]{Proposition}

\theoremstyle{definition} 
\newtheorem{defn}{Definition}

\usepackage{fullpage}
\newcommand{\rk}{\text{rk }}

\title{The dihedral genus of a knot}
\author{Patricia Cahn and Alexandra Kjuchukova}

\thanks{This work was supported by NSF-DMS grant 1821212 to P. Cahn and 1821257 to A. Kjuchukova}
\begin{document}
\maketitle

\begin{abstract} 
 Let $K\subset S^3$ be a Fox $p$-colored knot and assume $K$ bounds a locally flat surface $S\subset B^4$ over which the given $p$-coloring extends. This coloring of $S$ induces a dihedral branched cover $X\to S^4$. Its branching set is a closed surface embedded in $S^4$ locally flatly away from one singularity whose link is $K$. When $S$ is homotopy ribbon and $X$ a definite four-manifold, a condition relating the signature of $X$ and the Murasugi signature of $K$ guarantees that $S$ in fact realizes the four-genus of $K$.  We exhibit an infinite family of knots $K_m$ with this property, each with a {Fox 3-}colored surface of minimal genus $m$.  As a consequence, we classify the signatures of manifolds $X$ which arise as dihedral covers of $S^4$ in the above sense.\\

\end{abstract}

\section{Introduction}

The Slice-Ribbon Conjecture of Fox~\cite{fox1962problems}  asks whether every smoothly slice knot in $S^3$ bounds a ribbon disk in the four-ball. The analogous question can be asked in the topological category, namely: does every topologically slice knot bound a locally flat homotopy ribbon disk in $B^4$? Recall that a properly embedded surface with boundary $F'\subset B^4$ is {\it homotopy ribbon} if the fundamental group of its complement is generated by meridians of $\partial F'$ in $S^3$. Ribbon disks are easily seen to be homotopy ribbon whereas homotopy ribbon disks need not be smooth. 

For knots of higher genus, the generalized topological Slice-Ribbon Conjecture asks whether the topological four-genus of a knot is always realized by a homotopy ribbon surface in $B^4$. When a knot $K$ admits Fox $p$-colorings, we approach this problem by studying locally flat, oriented surfaces $F'\subset B^4$ with $\partial F'=K$ over which some $p$-coloring of $K$ extends, in the sense defined in Section~\ref{genusdefs}. The minimal genus of such a surface, when one exists, we call the {\it $p$-dihedral genus of $K$}. 

When $K$ is slice and $p$ square-free, it is classically known that the colored surface $F'$ for $K$ can always be chosen to be a disk. This is essentially a consequence of~\cite[Lemma~3]{casson1975gordon}; a detailed explanation can be found in~\cite[Lemma~9]{geske2018signatures}. Put differently, $p$-dihedral genus and classical four-genus coincide for slice knots. Furthermore, the topological Slice-Ribbon Conjecture is true for $p$-colorable slice knots if and only if the minimal $p$-dihedral genus for these knots can always be realized by homotopy ribbon surfaces.  With this in mind, given a square-free integer $p$ and a $p$-colorable knot $K$, we ask:

{\bf Question 1.} Is the (topological) four-genus of $K$ equal to its (topological) $p$-dihedral genus?

{\bf Question 2.} Is the $p$-dihedral genus of $K$ realized by a homotopy ribbon surface? 

When both of these questions are answered in the affirmative for a knot $K$ with respect to some integer $p$, it follows that the topological four-genus and homotopy ribbon genus of $K$ are equal; that is, the generalized topological Slice-Ribbon Conjecture holds for $K$.
If $K$ is not slice, requiring that it satisfies Questions 1 and 2 is a priori a stronger condition than satisfying the generalized Slice-Ribbon Conjecture; however, the advantage of this point of view is that dihedral genus can be studied using dihedral branched covers. 

Specifically, our approach is the following. Start with a branched cover of $f': X'\to B^4$ branched along a locally flat properly embedded surface $F'$ with $\partial F'=K$; that is, $F'$ is a properly embedded topological submanifold of $B^4$. We now construct a new cover $f: X\to S^4$ by taking the cones of $\partial X'$, $S^3$ and the map $f$. The branching set of $f$ is a surface $F$ embedded in $S^4$ locally flatly except for one singular point whose link is $K$. {Depending on the knot $K$ and the map $f$, it may be the case that this construction yields a total space $X$ that is again a manifold. In general, $X$ has one singular (non-manifold) point $\mathfrak{z}$, the pre-image of the singularity on $F$. The link of $\mathfrak{z}$ is the cover $f_|$ of $S^3$ branched along $K$. We will consider the signature of $X$ whether it is a manifold or has an isolated singularity. In the latter case, by the signature of $X$ we mean the Novikov signature of the manifold with boundary obtained by deleting an open neighborhood of $\mathfrak{z}$ in $X$.}

When { $f: X\to S^4$ in the above construction} is a $p$-fold {\it irregular dihedral cover} (Definition~\ref{dih}), an invariant of ($p$-colored) knots, $\Xi_p$, is extracted from this construction. This invariant is our main tool. In a general setting, $\Xi_p$ can be thought of as a defect term in the formula for the signature of a branched cover, resulting from the fact that the branching set is not locally flat. Put differently, the presence of a cone singularity $K$ on the branching set causes the signature of the cover to deviate from the smooth case by a term denoted $\Xi_p(K,\rho)$. This term depends only on the isotopy class of the knot $K$ and its Fox $p$-coloring $\rho$, but not on the locally flat part of the branching set.

Given a dihedral cover  $f: X\to S^4$ whose branching set is orientable with one singularity, we in fact have
\begin{equation}\label{xi-short}
\Xi_p(K,\rho)=-\sigma(X)	
\end{equation}
by~\cite[Theorem~1.4]{kjuchukova2018dihedral}, { when $X$ is a manifold, and 
\begin{equation}\label{novikov-xi}
	\Xi_p(K,\rho)=-\sigma(X', \partial X')
\end{equation}
when $X$ has a singularity~\cite[Theorem~7]{geske2018signatures}. 
In the latter formula, 
$\partial X'$ is the dihedral cover of $K$ induced by $f_|$, and $\sigma(\cdot, \cdot)$ denotes the Novikov signature of a manifold with boundary. Of course, the first formula for $\Xi_p(K,\rho)$ in terms of $X$ is a special case of the second, since the signature of a manifold is unchanged by deleting an open neighborhood of a point.}

{Unless explicitly stated otherwise, we will only consider orientable branching sets. Thus, we} take the above {signature} equation to be the definition of $\Xi_p(K,\rho)$.  In Equation~\ref{xi} we recall an explicit formula~\cite{kjuchukova2018dihedral}  for $\Xi_p$ which does not rely on constructing the cover $X$. We also note that  $\Xi_p(K,\rho)$ can be computed algorithmically from a colored diagram of $K$~\cite{cahnkjuchukova2018computing}. We often suppress notation and write $\Xi_p(K)$ when the choice of coloring is clear, or when a knot admits a unique $p$-coloring (up to permuting the colors). Thus, for a two-bridge knot $K$, we will write simply $\Xi_p(K)$.  
The main result of this paper, Theorem~\ref{main}, obtains a certain genus bound for $K$ from $\Xi_p(K)$.

As implied by the above, the signature defect $\Xi_p(K)$ is defined for a knot $K$ which arises as the only singularity on the branching set (not necessarily orientable) of an irregular dihedral cover~\cite{geske2018signatures}. A knot $K$ is called {\it $p$-admissible over $S^4$}, or simply {\it $p$-admissible}, if there exists a $p$-fold dihedral cover  $f: X\to S^4$ whose branching set is embedded and locally flat except for one singularity whose link is $K$.  If, in addition, the covering space $X$ is a topological manifold, $K$ is called {\it strongly $p$-admissible}. The distinguishing property of {\it strongly} $p$-admissible knots\footnote{Like the invariant $\Xi_p$, the notion of (strong) $p$-admissibility of a knot may depend on the choice of coloring. We do not dwell on this presently since all examples in this paper are two-bridge knots and their colorings are unique.} is that their dihedral covers are $S^3$. Admissibility of knots is studied in~\cite{kjorr2017admissible}.

In {Section ~\ref{genusdefs.sec}}, we put side by side the relevant notions of knot four-genus, recall several definitions, and state our main results: {Theorems~\ref{bound}, ~\ref{example} and ~\ref{sum}}. In {Theorem ~\ref{bound}}, we give a lower bound on the homotopy-ribbon $p$-dihedral genus of a colored knot $K$ in terms of the invariant $\Xi_p(K)$. We also give a sufficient condition for when this bound is sharp.  

In Theorem~\ref{example} and Theorem~\ref{sum}, we construct, for any integer $m\geq 0$, infinite families of knots for which the 3-dihedral genus and the topological four-genus are both equal to $m$. The basis of this construction are the knots $K_m$ pictured in Figure~\ref{knots}.  The various four-genera of these knots are computed with the help of Theorem~\ref{bound}.  In particular, for these knots, the lower bound on genus obtained via branched covers is exact and the generalized topological Slice-Ribbon Conjecture is seen to hold. {The proofs of Theorems 1, 2 and 3 are given in Section ~\ref{proof}.} 
 
The technique we apply is the following.  Given a strongly $p$-admissible knot $K$, one can evaluate $\Xi_p(K)$ by realizing $K$ as the only singularity on the branch surface of a dihedral cover of $S^4$.  Each of the knots $K_m$ arises as the only singularity on the branching set of a 3-fold dihedral cover $$f_m: \#(2m+1)\overline{\mathbb{C}P}^2\rightarrow S^4.$$ The branching set of $f_m$ is the boundary union of the cone on $K_m$ with the surface $F_m'$ realizing the four-genus of $K_m$. We construct these covering maps explicitly using singular triplane diagrams, a technique introduced in~\cite{cahnkjuchukova2017singbranchedcovers}.  Equivalently, we construct a family of covers $\#(2m+1)\mathbb{C}P^2\rightarrow S^4$, again with oriented, connected branching sets, with the mirror images of the knots $K_m$ as singularities. This construction appears in Section ~\ref{trisections}.  As a corollary of this construction, we realize all odd integers as values of $\Xi_3$. {In Theorem ~\ref{range}, we prove that the range of values of $\Xi_3$ on strongly admissible knots is precisely the set of odd integers.}

We work in the topological category, except where explicitly stated otherwise. Throughout, $F$ denotes a closed, connected, oriented surface, and $F'$ a connected, oriented surface with boundary. $D_p$ denotes the dihedral group of order $2p$, and $p$ is always assumed odd. 

{\bf Acknowledgments.} The idea of dihedral genus first appeared in discussion with Kent Orr~\cite{kjorr2017admissible}. The examples used in Theorem~\ref{example} were inspired by joint work with Ryan Blair~\cite{blair2018simply}.

\section{Dihedral four-genus and the {main theorems}}\label{genusdefs.sec}

\subsection{Some old and new notions of knot genus}
\label{genusdefs}

We study the interplay between the following notions of four-genus for a Fox $p$-colorable knot $K\subset S^3$.  Classically, the {\it smooth (resp. topological) four-genus} is the minimum genus of a smooth (resp. locally flat) embedded orientable surface in $B^4$ with boundary $K$.  The {\it smooth (resp. topological) $p$-dihedral genus} of a $p$-colored knot $K$ is, informally, the minimum genus of such a surface $F'$ in $B^4$ over which the $p$-coloring of $K$ extends. Precisely, a given $p$-coloring $\rho$ of $K$ extends over $F'$ if there exists a homomorphism $\bar{\rho}$ which makes the following diagram commute (where $i_\ast$ the map induced by inclusion):

\begin{center}
\begin{tikzcd}
	
		\pi_1(S^3-K)\arrow[r,"i_*"]\arrow[d,"\rho"]&\pi_1(B^4-F')\arrow [ld, dashed, "\bar{\rho}"]\\
		D_{p}
		\end{tikzcd}
		\end{center}

The $p$-dihedral genus above is defined for a knot $K$ with a fixed coloring $\rho$, hence we denote it $g_p(K,\rho)$ in the topological case.  We define the $p$-dihedral genus of a {$p$-colorable} knot $K$ to be the minimum $p$-dihedral genus of $K$ over all $p$-colorings $\rho$ of $K$, and denote this by $g_p(K)$ in the topological case.  Note that not every $p$-colored  knot $K$ admits a surface $F'$ as above. In~\cite{kjorr2017admissible}, we determine a necessary and sufficient condition for the existence of a connected oriented surface that fits into this diagram.  {When there is no surface over which a given coloring $\rho$ of $K$ extends, we define $g_p(K,\rho)$ to be infinite, and similarly for the refined notions of dihedral genus defined below.}

The {\it ribbon genus} of $K$ is the minimum genus of a smooth embedded orientable surface $F'$ in $B^4$ with boundary $K$, such that $F'$ has only local minima and saddles with respect to the radial height function on $B^4$.  The {\it smooth (topological) homotopy ribbon genus} of a knot $K$ is the minimum genus of a smooth (locally flat) embedded orientable surface $F'$ in $B^4$ with boundary $K$ such that $i_\ast: \pi_1(S^3-K)\twoheadrightarrow \pi_1(B^4-F')$, that is, inclusion of the boundary into the surface complement induces a surjection on fundamental groups. Finally, given a $p$-colorable or $p$-colored knot, its {\it ribbon $p$-dihedral genus} or {\it smooth (topological) homotopy ribbon $p$-dihedral genus} are defined in the obvious way.  Observe that all notions of dihedral genus refer to surfaces embedded in the four-ball, even though ``four" is not among the multitude of qualifiers we inevitably use.

As a straight-forward consequence of the definitions, the following inequalities hold among the {\it smooth} four-genera of a knot: 

\begin{center}
\begin{tikzcd}
		\text{four-genus}\arrow[d,"\leq"]\arrow[r,"\leq"]&\text{hom. ribbon genus}\arrow[d,"\leq"]\arrow[r,"\leq"]&\text{ribbon genus}\arrow[d,"\leq"]\\
		p\text{-dihedral genus}\arrow[r,"\leq"]&p\text{-dihedral hom. ribbon genus}\arrow[r,"\leq"]&p\text{-dihedral ribbon genus}\\
		\end{tikzcd}
		\end{center}

Excluding the last column, the inequalities make sense and hold in the topological category too.

\subsection{The Main {Theorems}}\label{mainthm.sec}

Denote by $g_4(K)$ the topological 4-genus of a knot $K$, and by $\mathfrak{g}_{{p}}(K,\rho)$ the topological homotopy-ribbon $p$-dihedral genus of a knot $K$ with coloring $\rho$. Again, the minimum such genus over all colorings $\rho$ of $K$ is $\mathfrak{g}_{{p}}(K)$. Let $\sigma(K)$ be the (Murasugi) signature of the knot $K$. {We relate $\mathfrak{g}_{{p}}(K, \rho)$, $\Xi_p(K, \rho)$ and $\sigma(K)$.} Here, $\Xi_p(K, \rho)$ denotes the invariant discussed in the Introduction; { it is reviewed in more detail in this section and, in particular, we recall that it} can be computed using Equation \ref{xi}. 

\begin{thm}  
\label{bound}\label{main}
(A) {Let $K$ be a $p$-admissible knot with $p$-coloring $\rho$ and denote by $M$ the irregular dihedral cover of $K$ determined by $\rho$. The following inequality holds:}

	\begin{equation}\label{inequality}
	\mathfrak{g}_{{p}}(K, \rho)\geq \dfrac{|\Xi_p(K,\rho)| - \rk H_1(M; \mathbb{Z})}{p-1}-\dfrac{1}{2}
	\end{equation}

(B)	Let $K$ be a $p$-admissible knot and $F'\subset B^4$ a {locally flat} homotopy ribbon oriented surface for $K$ over which a given $p$-coloring $\rho$ of $K$ extends.{ Denote by $c(K)$ the cone on $K$, seen as embedded in $D^4=c(S^3)$.} If the associated singular dihedral cover of $S^4$ branched along $F'\cup_K c(K)$ is a definite manifold, then the inequality~(\ref{inequality}) is sharp. In particular, $F'$ realizes the dihedral genus $\mathfrak{g}_{{p}}(K, \rho)$ of $K$. If, in addition, the equality $$|\sigma(K)| = \frac{2|\Xi_p(K, \rho)|}{p-1} -1$$ holds, then the topological four-genus and the topological homotopy ribbon $p$-dihedral genus of $K$ coincide and equal $\frac{|\sigma(K)|}{2}$, so the generalized topological Slice-Ribbon Conjecture holds for $K$.

\end{thm}
 
  {\bf Remark.} If $K$ has multiple $p$-colorings, denote by $\text{min}_p(K)$ the minimum value of $$|\Xi_p(K,\rho)|-\rk H_1(M;\mathbb{Z})$$ over all such colorings of $K$. Theorem~\ref{main} implies:
 		 \begin{equation}\label{inequality2}
	\mathfrak{g}_{{p}}(K )\geq \dfrac{\text{min}_p(K)}{p-1}-\dfrac{1}{2}.
	\end{equation}

\begin{thm}\label{example} For every integer $m\geq 0$, there exists a knot $K_{m}$ and corresponding 3-coloring $\rho_{m}$, such that: 
	$$g_4(K_{m})=\mathfrak{g}_{{3}}(K_{m})=\dfrac{|\Xi_3(K_{m}, \rho_{m})|}{2} -\frac{1}{2}=m.$$ 
	That is, the inequality~(\ref{inequality2}) is sharp for these knots and computes their $3$-dihedral genus as well as their topological four-genus. The generalized Slice-Ribbon Conjecture holds for these knots.
\end{thm}

\begin{figure}[htbp]
	\includegraphics[width=3.5in]{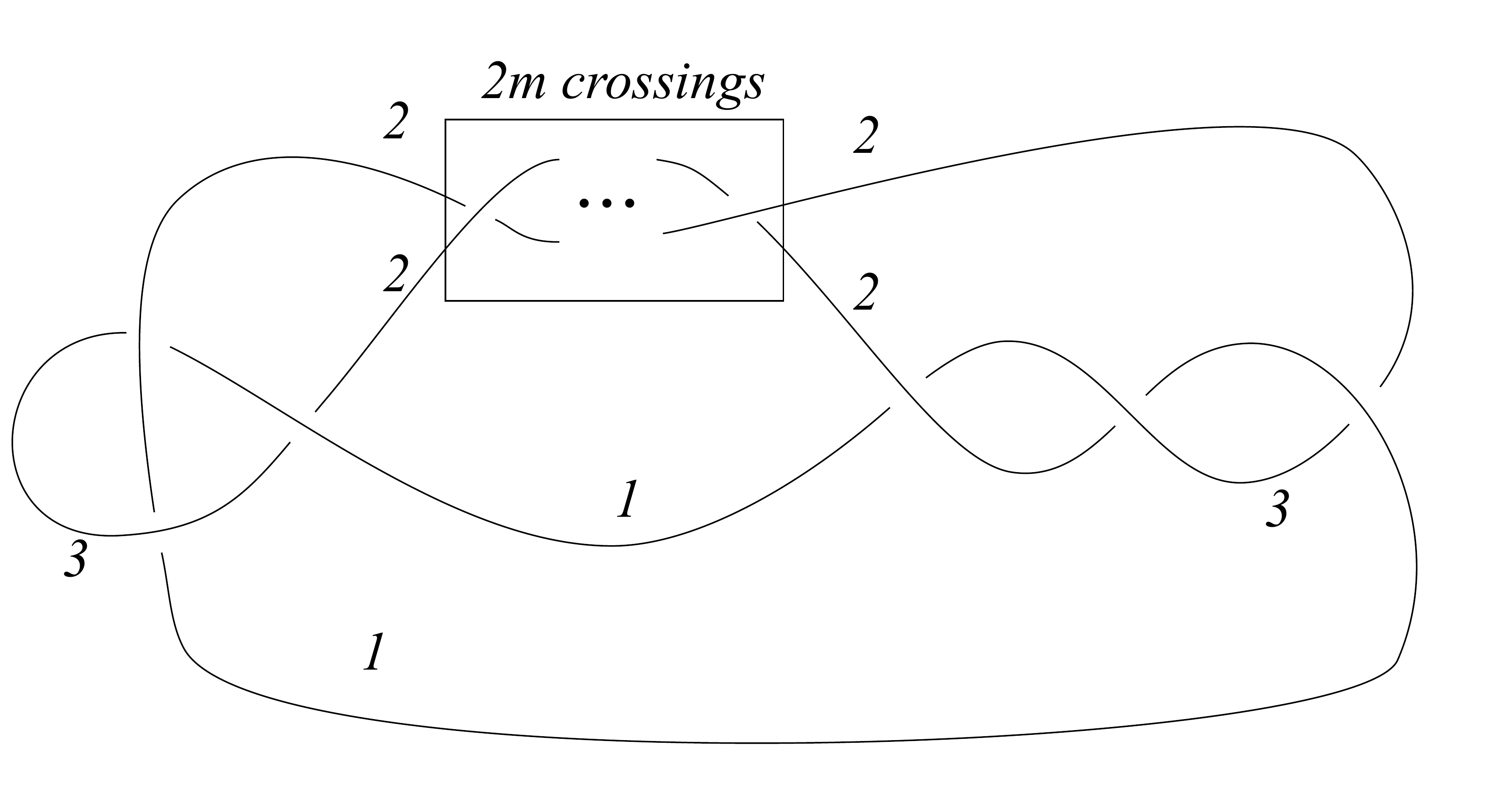}
	\caption{The knot $K_{m}$, $m\geq 0$, and its 3-coloring. We have $K_{0}=6_1$, $K_{1}=8_{11}$, $K_{2}=10_{21}$ and $K_{3}=12a723$.}
	\label{knots}
\end{figure}
\begin{thm}\label{sum}
	For any integer $m\geq 0$, there exist infinite families of knots whose 3-dihedral genus and topological four-genus are both equal to $m$.\end{thm}
	
\subsection{Singular dihedral covers of $S^4$ and the invariant $\Xi_p$}
\label{dihedraldefs}

In this section, we revisit the definition of a singular branched cover, and dihedral covers in particular.  We also review the context in which the invariant $\Xi_p$ arises, as well as a couple of techniques for its calculation.

\begin{defn}
\label{dih}
Let $Y$ be a manifold and $B\subset Y$ a codimension-two submanifold with the property that there exists a surjection $\varphi: \pi_1(Y-B) \twoheadrightarrow D_p$. Denote by $\mathring{X}$ the covering space of $Y-B$ corresponding to the conjugacy class of subgroups $\varphi^{-1}(\mathbb{Z}/2\mathbb{Z})$ in $\pi_1(Y-B)$, where $\mathbb{Z}/2\mathbb{Z}\subset D_p$ is any reflection subgroup. The completion of $\mathring{X}$  to a branched cover $f: X\to Y$ is called the {\it irregular dihedral} $p$-fold cover of $Y$ branched along $B$.
\end{defn}

The manifolds whose irregular dihedral covers we will consider are $S^3$, $B^4$ and $S^4$. {The $\Xi_p$ invariant was originally defined in the more general context of a dihedral cover of an arbitrary four-manifold $Y$ with a singularly embedded branching set~\cite{kjuchukova2018dihedral}.}

{Recall the following construction from the introduction.} Let $F'$ be a surface with connected boundary $K$, properly embedded in $B^4$ and locally flat. Given a branched cover of manifolds with boundary $f': X'\to B^4$, one constructs a {\it singular} branched cover of $S^4$ by coning off $\partial X'$, $\partial B^4$ and the map $f'$. The resulting covering map, $f: X\to S^4$, has total space $X:=X'\cup_{\partial X'} c(\partial X')$, where $c(\partial X')$ denotes the cone on $\partial X'$. The branching set is a closed surface $F:=F'\cup_K c(K)$ embedded in $S^4$ with a singularity (the cone point) whose link is $K$. The space $X$ obtained in this way is a manifold if and only if $\partial X'\cong S^3$. 

{Denote by $\sigma(X',\partial X')$ the Novikov signature of the manifold with boundary given as a cover of $B^4$ branched along $F'$. When $\partial X' = S^3$, denote by $\sigma(X)$ the signature of the manifold $X$. In this case, we have $\sigma(X) = \sigma(X', \partial X')$. }

{Given $f': X'\to B^4$ as before with $f'$ an (irregular) dihedral covering map, we always assume that the associated homomorphism $\rho: \pi_1(S^3-K)\to D_p$ is surjective or, equivalently, that $\partial X'$ is connected. In this case, assuming $F'$ is orientable, $\Xi_p(K, \rho)=-\sigma(X', \partial X')$ by~\cite[Theorem~7]{geske2018signatures}. In particular, when $X$ is a manifold, this equation reduces to the earlier result $\Xi_p(K, \rho)=-\sigma(X)$~\cite[Theorem~1.4]{kjuchukova2018dihedral}.}
 
 { Below, we recall two formulas for $\Xi_p(K, \rho)$ from~\cite{kjuchukova2018dihedral}. Equation~\ref{xi} allows for $\Xi_p(K, \rho)$ to be computed in terms of $K$ and its coloring using~\cite{cahnkjuchukova2018computing}  and~\cite{cahnkjuchukova2016linking}. Equation~\ref{xisignature} expresses $\Xi_p(K, \rho)$ in terms of a singular branched cover of $S^4$ in the more general case where the branching set is a possibly non-orientable surface.}

{Refocusing for a moment on the case where the dihedral branched cover $X$ of $S^4$ is a manifold, we note that} there exist many infinite families of knots $K\subset S^3$ whose irregular dihedral covers are homeomorphic to $S^3$. For example, this is a property shared by all $p$-colorable two-bridge knots (a well-known fact recalled in the proof of~\cite[Lemma~3.3]{kjuchukova2018dihedral}). By definition, if a $p$-admissible knot $K$ has $S^3$ as its dihedral cover, then it is in fact strongly $p$-admissible. 
We are then able study invariants of $K$ using four-dimensional techniques such as trisections.   
Criteria for admissibility of singularities are {discussed in more detail} in~\cite{cahnkjuchukova2017singbranchedcovers} where we also use the invariant $\Xi_p(K)$ to give a homotopy ribbon obstruction for {strongly} $p$-admissible knots $K$. 
A generalization of this ribbon obstruction to all $p$-admissible knots appears in~\cite{geske2018signatures}.

We conclude this section by reviewing the formula for computing the invariant $\Xi_p$ given in~\cite{kjuchukova2018dihedral}. Let $p$ be an odd integer and $K$ a $p$-admissible knot. Let $V$ be a Seifert surface for $K$ and $V^\circ$ the interior of $V$. Denote by $\beta\subset V^\circ$ a mod~$p$ characteristic knot\footnote{{Precisely, if $K$ admits multiple $p$-colorings, one must work with a characteristic knot corresponding to the coloring in question. The sense in which a characteristic knot determines a coloring is laid out in~\cite[Proposition~1.1]{CS1984linking}.  The examples we construct always admit a unique $p$-coloring, up to permuting the colors, and therefore a unique equivalence class of mod~$p$ characteristic knots.}} for $K$, as defined in~\cite{CS1984linking}. 
Also denote by $L_V$ the symmetrized linking form for $V$ and by $\sigma_{\zeta^i}$ the Tristram-Levine $\zeta^i$-signature, where $\zeta$ is a primitive $p^{\text{th}}$ root of unity.  Finally, let $W(K,\beta)$ be the cobordism constructed in~\cite{CS1984linking} between the $p$-fold cyclic cover of $S^3$ branched along $\beta$ and the $p$-fold dihedral cover of $S^3$ branched along $K$ and determined by $\rho$. We briefly describe the manifold $W(K,\beta)$.  
Let $\Sigma$ be the $p$-fold cyclic branched cover of $\beta$ and let $\Sigma_p(\beta)\times[0, 1]\to S^3\times[0, 1]$ be the induced cyclic cover branched along $\beta\times[0, 1]$. Letting $\mathbb{Z}/2\mathbb{Z}$ act on an appropriate subset of $\Sigma_p(\beta)\times\{0\}$, one obtains $W(K,\beta)$ as a quotient of $\Sigma_p(\beta)\times[0, 1]$ by this action. One boundary of this quotient, namely $\Sigma_p(\beta)\times\{1\}$, is clearly the $p$-fold cyclic cover of $\beta$. The other boundary component, that is, the image of $\Sigma_p(\beta)\times\{0\}$ under the $\mathbb{Z}/2\mathbb{Z}$ action, is the dihedral cover of $\alpha$ as shown in~\cite[Proposition~1.1]{CS1984linking}.   By~\cite[Theorem~1.4]{kjuchukova2018dihedral}, 

\begin{equation}
\label{xi}
	\Xi_p(K, \rho)=\dfrac{p^2-1}{6p}L_V(\beta, \beta)+\sigma(W(K, \beta))+\sum_{i=1}^{p-1}\sigma_{\zeta^i}(\beta).
\end{equation}

The Novikov signature $\sigma(W(K,\beta))$ can be computed in terms of linking numbers in the dihedral cover of $K$~\cite[Proposition~2.5]{kjuchukova2018dihedral}. Thus, the above formula allows $\Xi_p(K)$ to be evaluated directly from a $p$-colored diagram of $K$, without direct reference to a four-dimensional construction. An explicit algorithm for performing this computation is outlined in~\cite{cahnkjuchukova2018computing}.  Note also that when a knot $K$ is realized as the only singularity on an embedded surface $F\subset S^4$ and moreover this surface is {presented by} a Fox $p$-colored singular triplane diagram, \cite{cahnkjuchukova2017singbranchedcovers} gives a method for computing $\Xi_p(K)$ from this data, via the signature of the associated cover of $S^4$. This technique is reviewed and applied in Section~\ref{trisections} below.\\

We also {review the context in which Equations~\ref{xi-short} and~\ref{novikov-xi} arise, allowing us to relate} $\Xi_p(K,\rho)$ to the signature of a singular branched cover $X$ of $S^4$. Consider an irregular dihedral cover  $f: X\to S^4$ whose branching set $F$ is an embedded surface, {\it not necessarily orientable,}  locally flat away from one singularity $z\in F$ of type $K$. The induced coloring of $F$ is{, as always,} an extension of $\rho$. {Once again we denote by $X'$ the dihedral cover of $B^4$ branched along the complement in $F$ of a neighborhood of the singular point $z$. Note also that $X'$ is obtained by deleting from $X$ a small open neighborhood of $f^{-1}(z)$}. We have
\begin{equation}
\label{xisignature}
\Xi_p(K)=-\dfrac{p-1}{4}e(F)-\sigma(X', \partial X'),
\end{equation}
where $e(F)$ denotes the self-intersection number of $F$. This is a special case of the signature formula for dihedral branched covers over an arbitrary base\footnote{{The reference~\cite{geske2018signatures} is written in the language of intersection homology. In the case of a singular branched cover $f: X\to S^4$, this is equivalent to the Novikov signature $\sigma(X', \partial X')$  since $X$ has only an isolated singularity.}} given in~\cite[Theorem 7]{geske2018signatures}. Note that, when $F$ is orientable {and $X$ a manifold, Equation~\ref{xisignature} reduces to Equation~\ref{xi-short}, that is, $\Xi_p(K)=-\sigma(X)$.  
In this case,  the $\Xi_p$ invariant of a singularity can be understood entirely in terms of the signature of the branched cover and, in particular, can be computed using four-manifold techniques.  We further} note that it is possible to realize {\it all} connected sums $\#n \mathbb{C}P^2$ as 3-fold dihedral covers of $S^4$ with one knot singularity on a connected, embedded branching set, if one allows the branching set to be non-orientable~\cite{blair2018simply}. {By contrast, we see in Theorem~\ref{range} that orientability of the branching set, together with a single singular point, imply that the signature of such a cover is odd. }

\section{Knots with equal topological and dihedral{ genera}}
\label{trisections} 

In this section we construct families of knots for which the topological, ribbon and 3-dihedral genus are equal. We use trisections of four-manifolds \cite{gaykirby2016trisections}, tri-plane diagrams \cite{meierzupan}, and singular tri-plane diagrams \cite{cahnkjuchukova2017singbranchedcovers}, all of which we review informally for the reader's convenience.

Given a smooth, oriented, 4-manifold $X$, a $(g,k_1,k_2,k_3)$-trisection of $X$ is a decomposition of $X=X_1\cup X_2\cup X_3$ into three 4-handlebodies with boundary, such that 
\begin{itemize}
	\item $X_i\cong \natural k_i (B^3\times S^1)$
	\item $X_1\cap X_2\cap X_3 \cong \Sigma_g$ is a closed, oriented surface of genus $g$
	\item $Y_{ij}=\partial (X_i\cup X_j)\cong \# k_l (S^2\times S^1)$ where $i$, $j$ and $l$ $\in \{1,2,3\}$ are distinct
	\item $\Sigma_g\subset Y_{ij}$ is a Heegaard surface for $Y_{ij}$.
\end{itemize}

Every embedded surface $F\subset S^4$ can be described combinatorially by a $(b;c_1,c_2,c_3)$-{\it tri-plane diagram} \cite{meierzupan}.  This is a set of three $b$-strand trivial tangles $(A,B, C)$, such that each boundary union of tangles  $A\cup \overline{B}$,  $B\cup \overline{C}$, and $C\cup \overline{A}$ is a $c_i$-component unlink, for $i=1,2,3$ respectively.  Here $\overline{T}$ denotes the mirror image of $T$.  To obtain $F$ from $(A,B,C)$, one views each of $A\cup \overline{B}$,  $B\cup \overline{C}$, and $C\cup \overline{A}$ as unlinks in bridge position in the spokes $Y_{12}$, $Y_{23}$, and $Y_{31}$ of the standard genus-0 trisection of $S^4$, glues $c_i$ disks to the components of each of these unlinks, and pushes these disks into the $X_i$ to obtain an embedded surface.

The authors introduce {\it singular tri-plane diagrams} and their colorings in \cite{cahnkjuchukova2017singbranchedcovers}.  A $(b; 1,c_2,c_3)$ singular tri-plane diagram is a triple of $b$-strand trivial tangles $(A,B,C)$.  As above $B\cup \overline{C}$  and $C\cup \overline{A}$ are $c_2$- and $c_3$-component unlinks.  $A\cup \overline{B}$ is a knot $K$.  To build a surface with one singularity of type $K$, one again views each of $A\cup \overline{B}$,  $B\cup \overline{C}$, and $C\cup \overline{A}$ in bridge position in the three spokes $Y_{12}$, $Y_{23}$, and $Y_{31}$ of the standard genus-0 trisection of $S^4$ and glues $c_2$ and $c_3$ disks to the components of each of the two unlinks.  Rather than glue disks to $A\cup \overline{B}$, one attaches the cone on $K$.  Note that by interchanging the order of the tangles $A$ and $B$, one obtains a surface with singularity $\overline{K}$, {the mirror of $K$}.

A {\it $p$-colored singular tri-plane diagram} is a singular tri-plane diagram together with an assignment of values in $\{1,2, \dots,p\}$ to the arcs of the diagram, such that on each tangle, the assignment is a Fox $p$-coloring, and such that the colors along the endpoints of each tangle agree.   Such a coloring induces a coloring on the corresponding singular surface.

\begin{figure}[htbp]\includegraphics[width=3.5in]{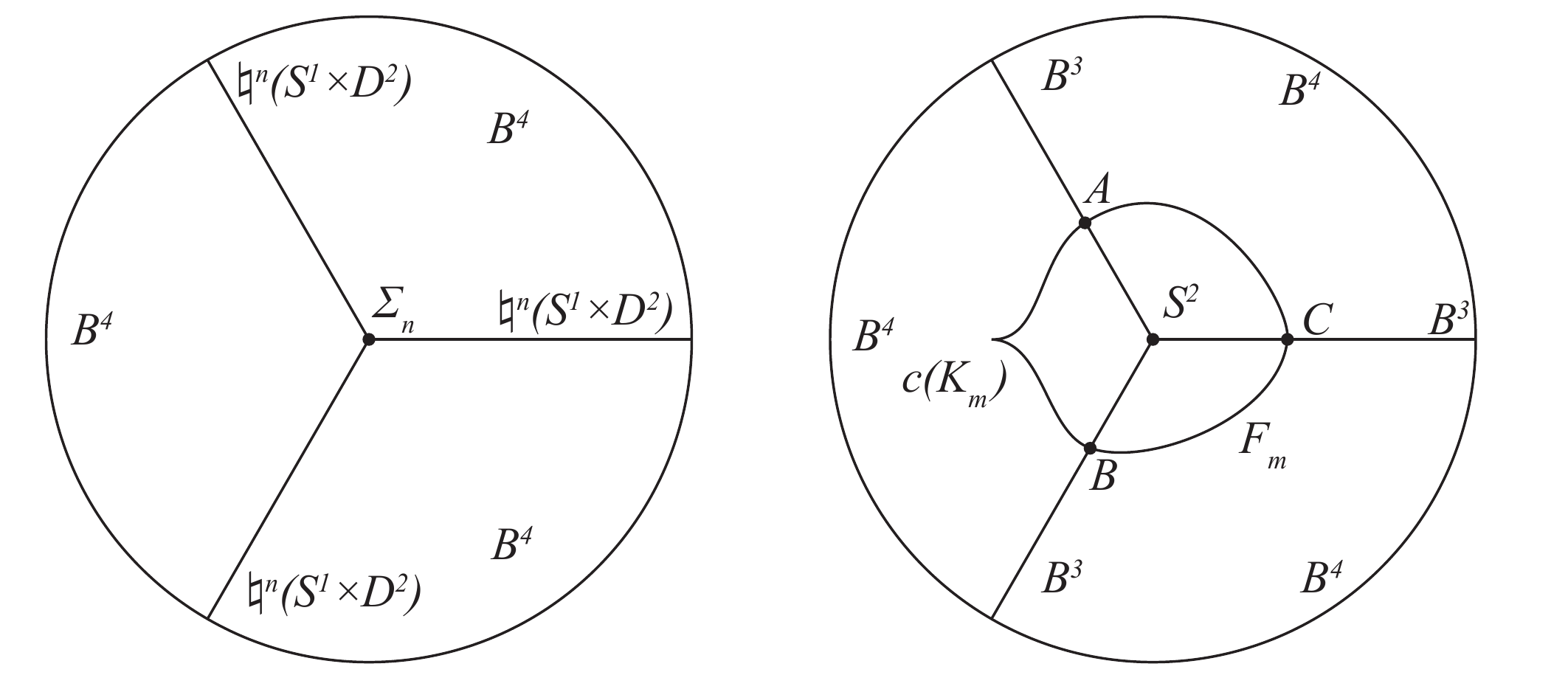}
\caption{A $(n;0,0,0)$-trisection of $\#n \overline{\mathbb{C}P}^2$, obtained as branched cover of $S^4$ over a trisected surface $F_m$ with one singularity $K_m$.}
\end{figure}

We use 3-colored singular triplane diagrams to construct a family of 3-fold dihedral covers of $S^4$ which realize the knots $K_m$ given in Figure~\ref{knots} as singularities on the branching sets.  This construction allows us to compute the values of $\Xi_3(K_m)$  using the induced trisections of the corresponding branched cover. {As a corollary, we obtain Theorem~\ref{range}, which establishes the range of the invariant~$\Xi_3$.}

\begin{prop}
	\label{covers}
	Each knot $K_m$ in Figure~\ref{knots} arises as the only singularity on  a 3-fold dihedral branched cover $f_m: \#(2m+1) \overline{\mathbb{CP}}^2\rightarrow S^4$ whose branching set $F_m$ is an oriented surface of genus $m$, embedded smoothly in $S^4$ away from the one singular point.  Equivalently, each knot $\overline{K}_m$ arises as the only singularity on  a 3-fold dihedral branched cover $\bar{f}_m: \#(2m+1)\mathbb{CP}^2\rightarrow S^4$, also with an embedded oriented branching set of genus $m$.
\end{prop}

{\bf Remark.}  By deleting a small neighborhood of the singularity on the branching set in $S^4$, one obtains an oriented, 3-colored surface in $F_m'\subset B^4$ with  $\partial F_m' = K_m$. In Section~\ref{proof}, we prove that the genus of $F_m'$ is minimal, that is, equal to $g_4(K_m)$.  Moreover, by construction, each surface $F'_m$ is ribbon.

\begin{proof}[Proof of Proposition~\ref{covers}]

We will construct the surface $F_m$ and will give its Fox coloring using a colored (singular) tri-plane diagram. From this information, we will produce a trisection of the dihedral cover of $S^4$ determined by this coloring. We will identify this cover as $\#n \overline{\mathbb{CP}}^2$, where $n=2m+1$. 	
	\begin{figure}[htbp]
	\includegraphics[width=\textwidth]{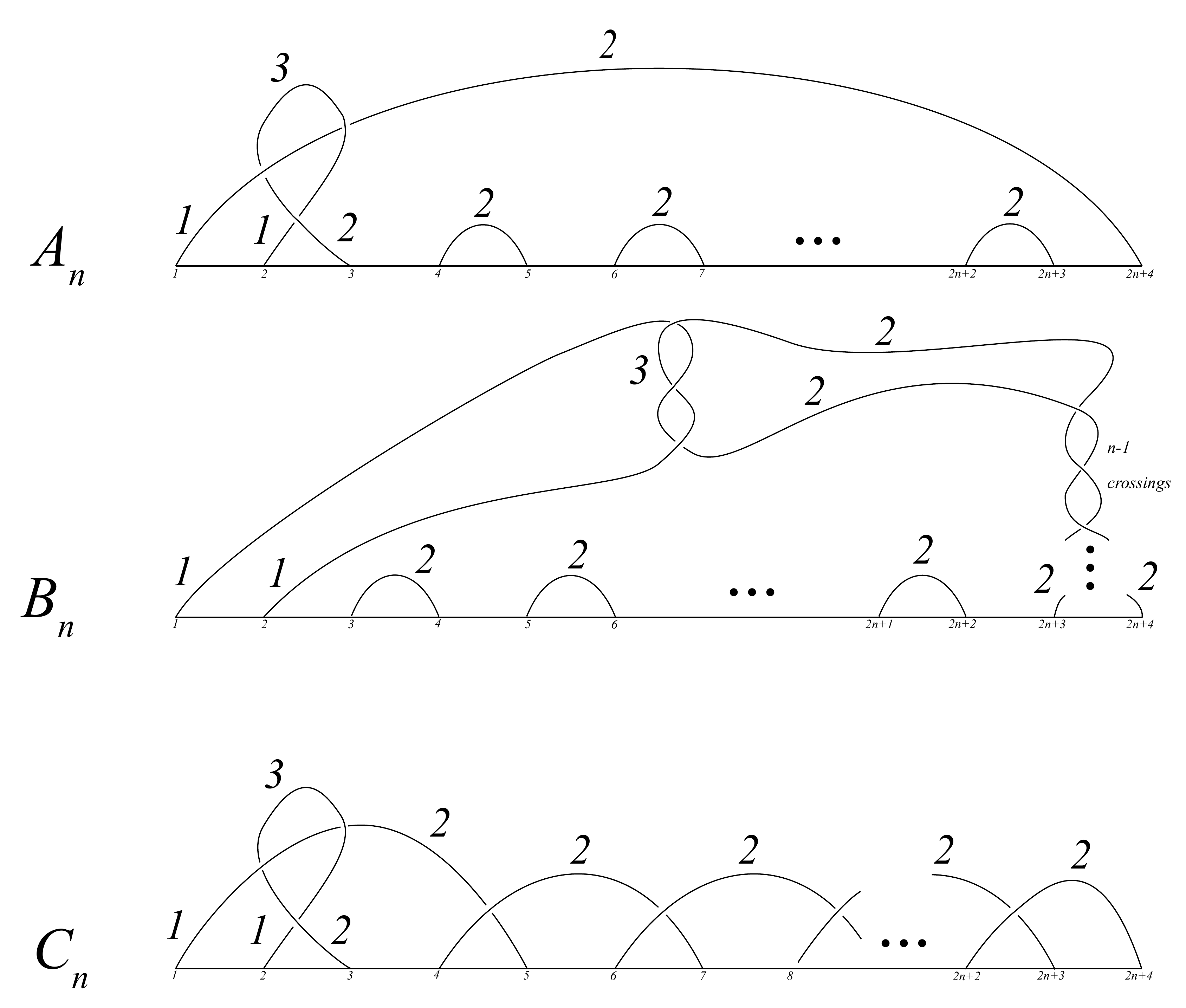}
	\caption{A colored tri-plane diagram corresponding to a branched covering $\#n\overline{\mathbb{C}P}^2\rightarrow S^4$, in the case where $n$ is odd.  The numbers $\{1,2,3\}$ along the arcs describe the coloring. There is one singularity $K_{m}$ on the branching set, where $m=(n-1)/2$.   By reversing the roles of $A_n$ and $B_n$, one obtains a branched covering $\#n\mathbb{C}P^2\rightarrow S^4$ with singularity $\overline{K}_m$. }
	\label{oddCP2triplane.fig}
	\end{figure}

The colored tri-plane diagram $(A_n,B_n, C_n)$ for $F_m$, where $m=(n-1)/2$, is shown in Figure ~\ref{oddCP2triplane.fig}.  We write the value $i\in\{1,2,3\}$ next to an arc of a tangle or knot if the homotopy class of the meridian of that arc is mapped to the reflection in $D_3$ fixing $i$.

The union $A_n\cup\overline{B}_n$ is the knot $\overline{K}_m$, while $B_n\cup \overline{C}_n$ and $C_n\cup \overline{A}_n$ are each 2-component unlinks; see Figure ~\ref{orientabilitycheck.fig} for a verification of this fact when $n=3$. A tri-plane diagram with $b$ bridges and $c_i$ components in each link diagram has Euler characteristic $c_1+c_2+c_3-b$; hence, the surface $F_m$ with singularity $\overline{K}_m$ has Euler characteristic $3-n$ and genus $m=(n-1)/2$ since $F_m$ is connected and orientable, and since the tangles $A_n$, $B_n$ and $C_n$ have $b=n+2$ bridges.

The fact that $F_m$ is orientable requires a careful check.  Consider the cell structure on $F_m$ corresponding to its tri-plane structure.  To show that $F_m$ is orientable, we show that it is possible to coherently orient the faces of this cell structure so that each edge (a bridge in one of the three tangles $A_n$, $B_n$, or $C_n$) inherits two different orientations from the two faces adjacent to it.  This is shown in Figure ~\ref{orientabilitycheck.fig} in the case $m=1$ (or $n=3$).

\begin{figure}[htbp]\includegraphics[width=\textwidth]{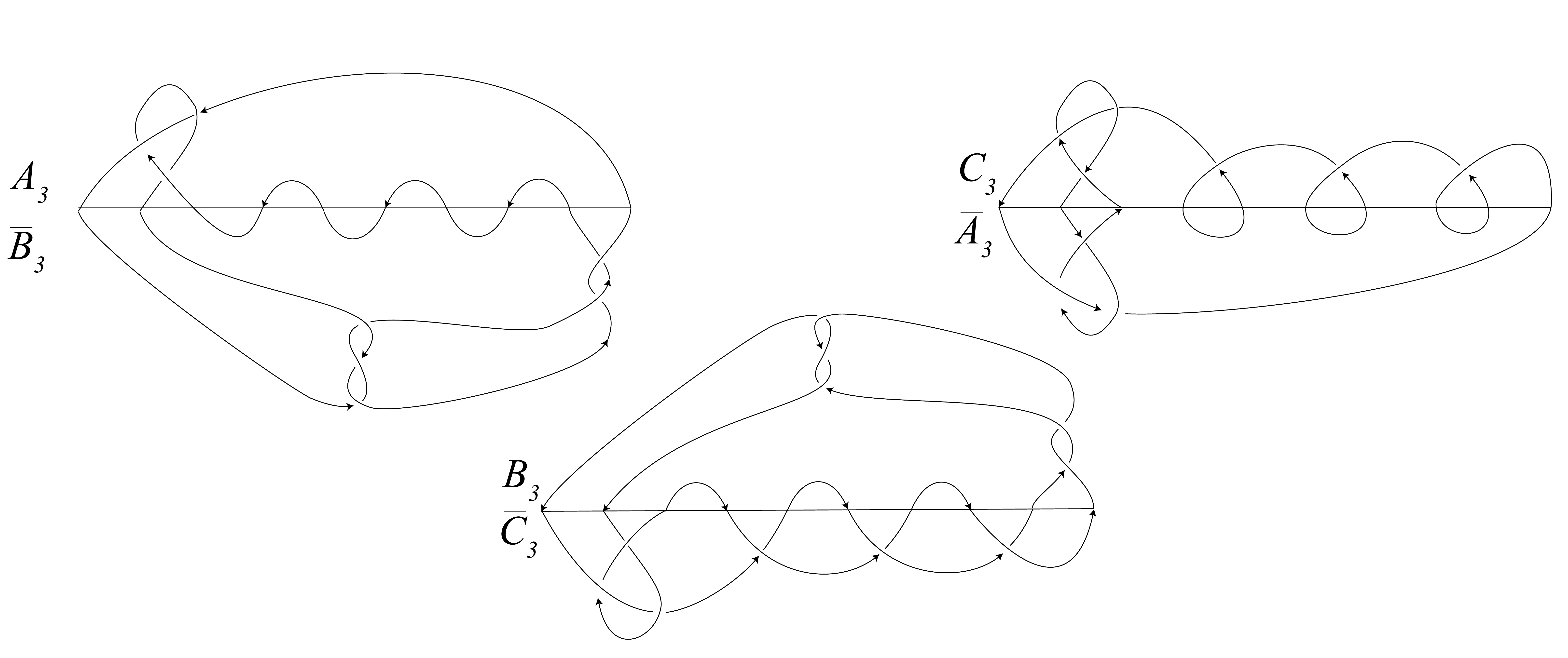}
	\caption{The links $A_3\cup \overline{B}_3$, $B_3\cup\overline{C}_3$, and $C_3\cup \overline{A}_3$. Note that $A_3\cup \overline{B}_3$ is the knot $\overline{K}_1$. }
	\label{orientabilitycheck.fig}
\end{figure}

An Euler characteristic computation shows that the 3-fold dihedral branched cover of the bridge sphere $S^2$, branched along the $2(n+2)$ endpoints of the bridges, is a surface $\Sigma_n$ of genus $n$. We now show this 3-colored tri-plane diagram $(A_n,B_n, C_n)$ gives rise to a genus $n$ trisection of $\#n\overline{\mathbb{C}P}^2$ with central surface $\Sigma_n$ following a method explained in~\cite{cahnkjuchukova2017singbranchedcovers}.  The branching set $F_m$ is orientable and has one singularity of type $K_m$, so it will follow from Equation \ref{xisignature} that $\Xi_3(K_m)=-\sigma(\#_n\overline{\mathbb{C}P}^2)=n$.

If a properly embedded $b$-strand tangle $(T,\partial T)\subset (B^3,S^2)$ with arcs $t_1,t_2,\dots t_b$ is trivial, then by definition there exists a collection of disjoint arcs $d_1,d_2,\dots, d_b$ in $S^2$ such that the boundary unions $t_i\cup d_i$ bound a collection of disjoint disks in $B_3$.  We refer to the $d_i$ as {\it disk bottoms}.  The existence of such a collection of disks is equivalent to the arcs of $T$ being simultaneously isotopic to a collection of disjoint arcs (the $d_i$) in $S^2$.

To determine the trisection diagram, we must first find the disk bottoms for the three tangles $A_n$, $B_n$ and $C_n$, then lift them from the bridge sphere $S^2$ to its irregular dihedral cover $\Sigma_n$. The curves in the trisection diagram are formed by certain lifts of these disk bottoms; we identify these lifts later.

The disk bottoms for each tangle $A_n$, $B_n$, and $C_n$ are depicted in Figure~\ref{CPnshadows.fig}, in the case $n=3$.  In Figure~\ref{oddCP2downstairsshadows.fig}, we draw just three of the disk bottoms for each of $A_n$ (blue), $B_n$ (red), and $C_n$ (green) on the same copy of $S^2$.

\begin{figure}[htbp]
	\includegraphics[width=3.5in]{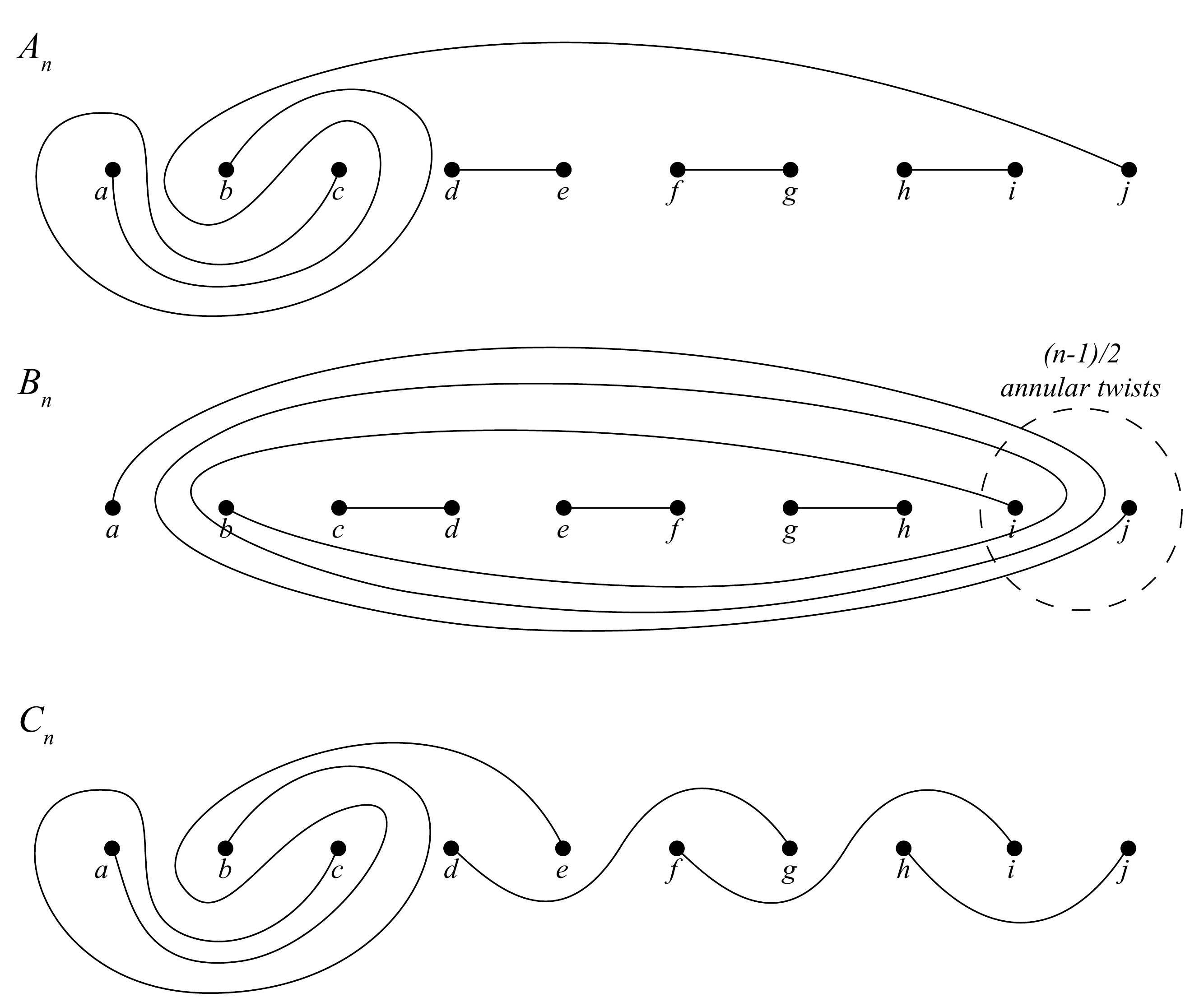}
	\caption{Disk bottoms for the tri-plane diagram $(A_n,B_n,C_n)$ when $n=3$.}\label{CPnshadows.fig}
\end{figure}

\begin{figure}[htbp]
	\includegraphics[width=6in]{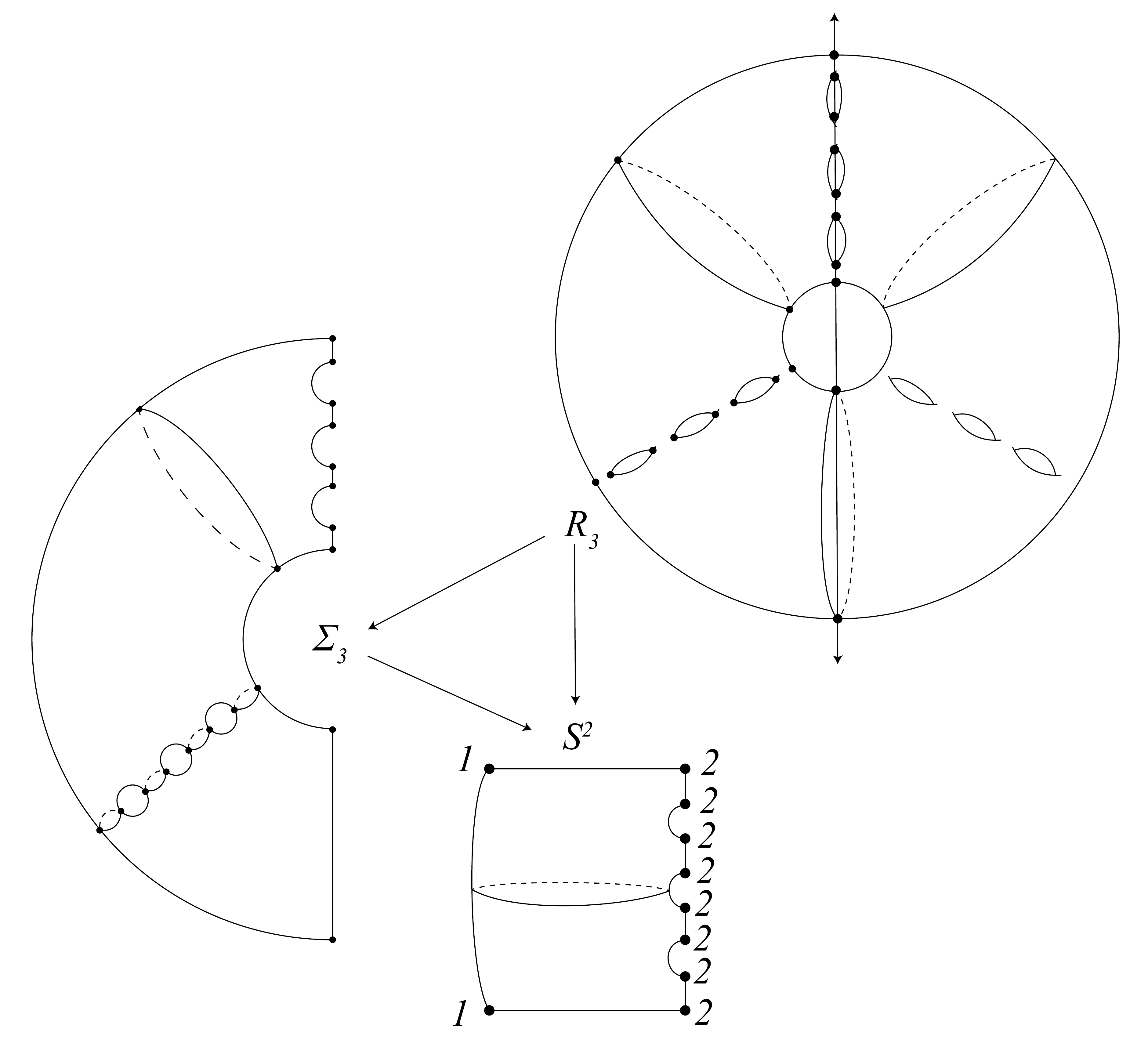}
	\caption{A 6-fold regular dihedral cover $R_n$ of $S^2$ branched along {$2(n+2)$} points, with $n=3$; the irregular cover $\Sigma_n$ is the quotient of $R_n$ by $180^\circ$ rotation about the vertical axis.}\label{hildenregularcover.fig}
\end{figure}

\begin{figure}[htbp]
	\includegraphics[width=4in]{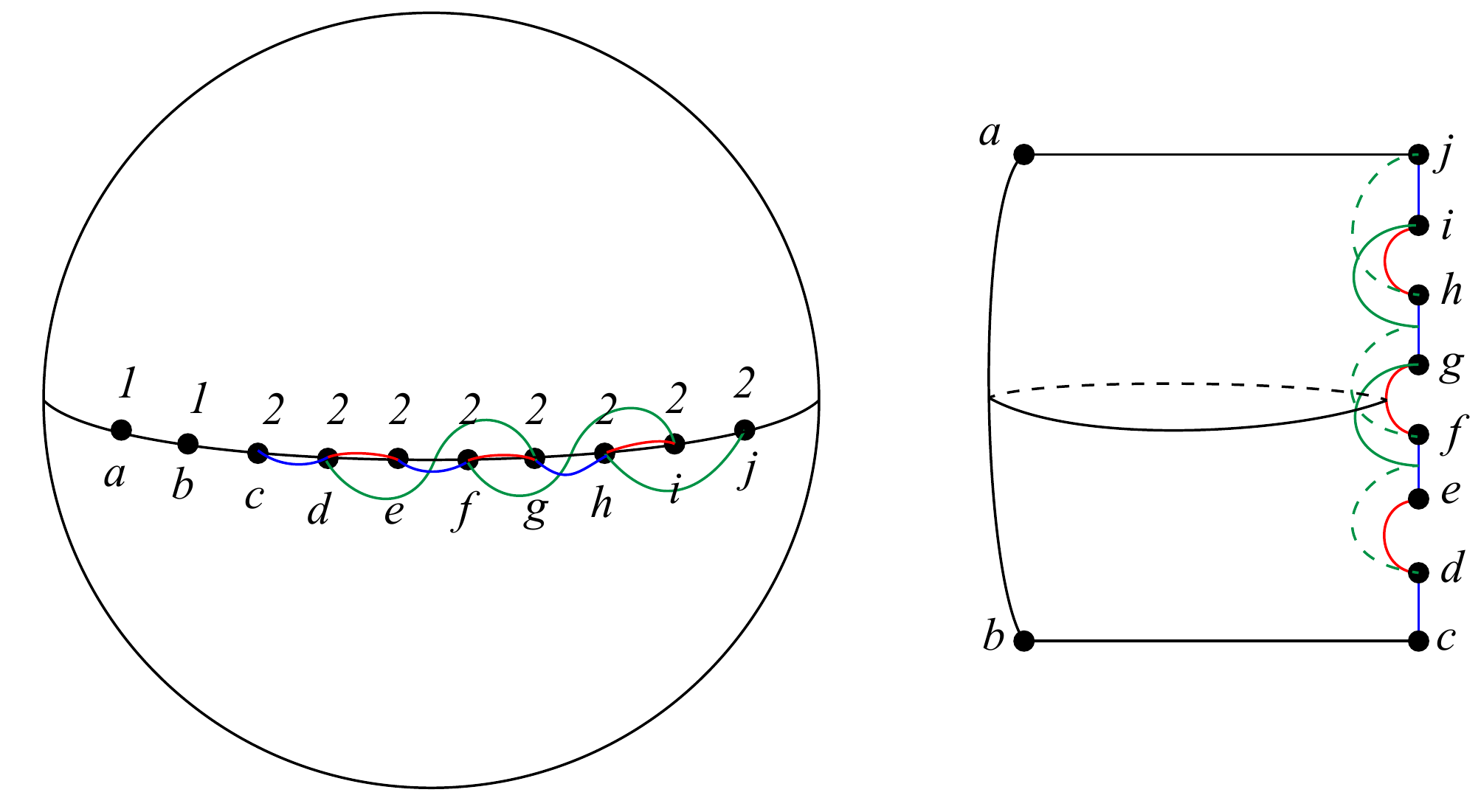}
	\caption{Disk bottoms for the tri-plane diagram $(A_n,B_n,C_n)$ when $n=3$, drawn on the bridge sphere.}\label{oddCP2downstairsshadows.fig}
\end{figure}

  In the next step of the proof, we use a construction of the irregular 3-fold dihedral cover $\Sigma_n\rightarrow S^2$, branched along {$2(n+2)$} points in $S^2$, due to Hilden \cite{hilden1974every}.  We review this construction now; the reader should refer to Figure ~\ref{hildenregularcover.fig} for an example in the case $n=3$. In this construction, the meridians of two branch points map to the transposition $(23)$ (equivalently, are colored `1'), and the meridians of the remaining {$2n+2$} branch points map to the transposition $(13)$ (equivalently, are colored `2').  One first constructs the 6-fold {\it regular} dihedral cover $R_n\rightarrow S^2$ branched along {$2(n+2)$} points determined by this coloring. The resulting surface has genus $3n+1$.  The 3-fold irregular dihedral cover $\Sigma_n$ is obtained from this regular one by an involution, namely $180^\circ$ rotation about the vertical axis.

 Next, we lift the disk bottoms from the bridge sphere to $\Sigma_n$, where $\Sigma_n$ is constructed as above. Each disk bottom has three lifts to $\Sigma_n$, two of which fit together to form a closed curve.  Not all of these closed curves are necessarily essential curves on $\Sigma_n$; see~\cite{cahnkjuchukova2017singbranchedcovers} for further examples.  However, we may choose $n-2$ disk bottoms for each tangle $(A_n,B_n,C_n)$ whose lifts are essential.  These lifts are shown in Figure ~\ref{hildencover.fig}, again in the case $n=3$.

\begin{figure}[htbp]
	\includegraphics[width=4in]{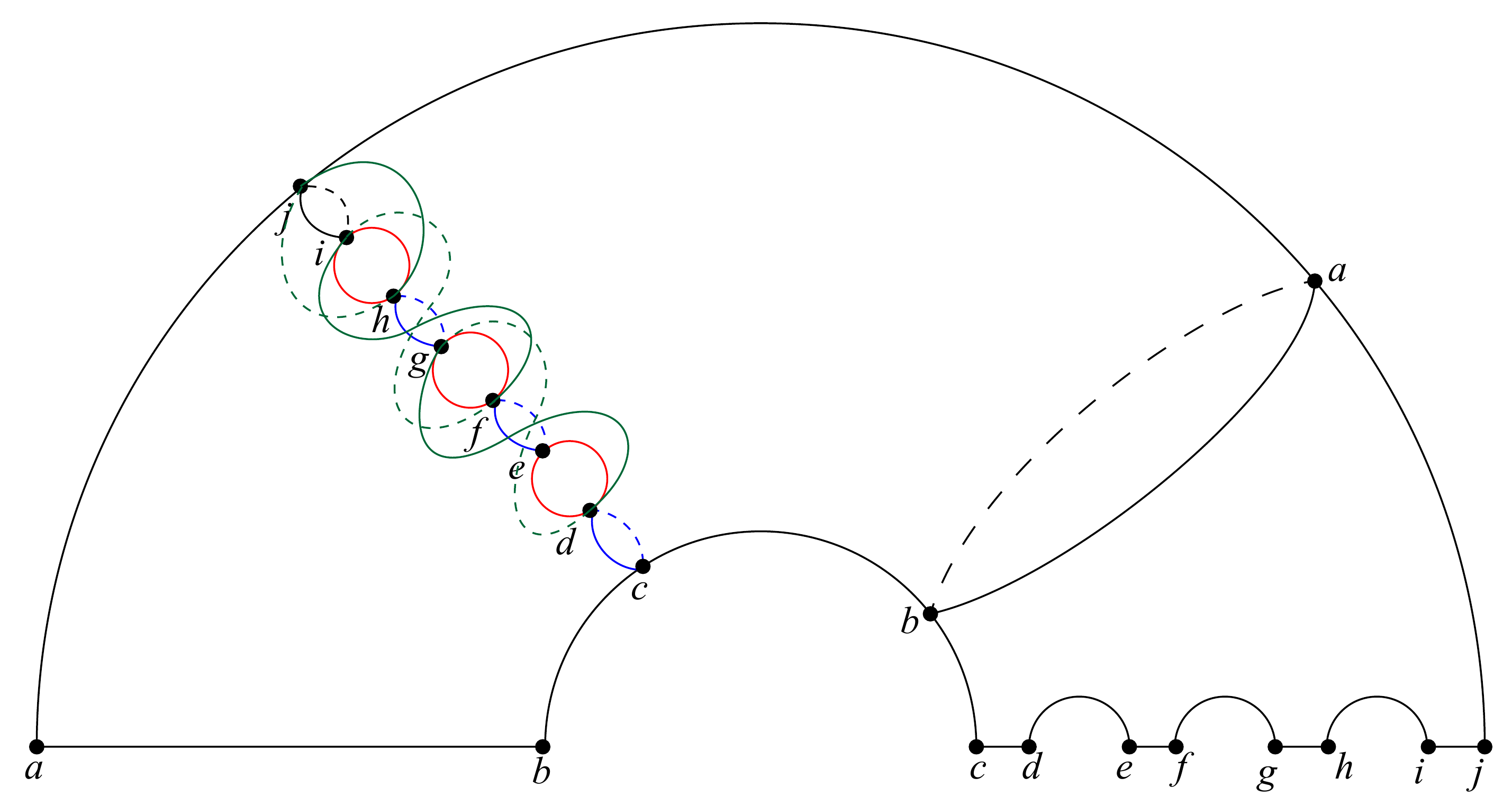}
	\caption{Lifts of disk bottoms to the 3-fold irregular dihedral cover of $S^2$, for the tri-plane diagram $(A_n,B_n,C_n)$, when $n=3$.}\label{hildencover.fig}
\end{figure}	

The resulting curves form a trisection diagram for $\#n\overline{\mathbb{C}P}^2$.  Moreover, the standard trisection of $S^4$, branched along $F_m$, lifts to a $(n;0,0,0)$-trisection of $\#n\overline{\mathbb{C}P}^2$.  This can be found by analyzing the lifts of the three pieces of the trisection of $(S^4,F_m)$; for details see Theorem 8 of \cite{cahnkjuchukova2017singbranchedcovers}.\end{proof}

We use the above construction to establish the range of the invariant $\Xi_3$.

\begin{thm} 
\label{range}
Let $n$ denote an integer. There exists a {{\it strongly}} 3-admissible singularity $K_n$ and a 3-coloring $\rho_n$ of $K_n$ such that  $\Xi_3(K_n, \rho_n)=n$ if and only if $n\in 2\mathbb{Z}+1$.
\end{thm}

{\bf Remark.} The proof of Theorem~\ref{range} is slightly more general than what the theorem statement requires. That is, we establish that $\Xi_p(K, \rho)$ is odd whenever $p\equiv 3\mod 4$. Realizability of all odd integers by $\Xi_p$ is open for $p\neq 3$.

\begin{proof}[Proof of Theorem~\ref{range}]
	We have given a construction realizing each of the knots $K_m$ as the only singularity on a branched cover $\#({2m+1}) \overline{\mathbb{CP}}^2\rightarrow S^4$ whose branching set is oriented. By Equation~\ref{xisignature}, it follows that $\Xi_3(K_m)=-\sigma(\#(2m+1)\overline{\mathbb{CP}}^2)=2m+1$, where $m\geq 0$. Note also that $\Xi_p(\overline{K}_m)=-\Xi_p(K_m)$ as proved in~\cite{cahnkjuchukova2017singbranchedcovers}, where $\overline{K}$ denotes the mirror image of $K$. Of course, $K$ is (strongly) $p$-admissible if and only if $\overline{K}$ is. This proves that all odd integers are contained in the range of the invariant $\Xi_3$ on {strongly}  3-admissible knots. 
	
Conversely, we will verify that for any $p$-coloring $\rho$ of any {strongly} $p$-admissible singularity $K$, the integer  $\Xi_p(K, \rho)$ is odd. It suffices to assume that $p\equiv 3\mod 4$. We use Equation~\ref{xi}. Since $p$ is odd, $p^2\equiv 1\mod4$, so $\dfrac{p^2-1}{6p}L_V(\beta, \beta)$ is even.  It follows from~\cite[Equation 2.20]{kjuchukova2018dihedral} that, if $p\equiv 3\mod 4$, the rank of $H_2(W(K,\beta); \mathbb{Z})$ is odd, hence so is the signature. Lastly, each $\sigma_{\zeta^i}$ is an even integer.  It follows that $\Xi_p(K)$ is odd. \end{proof}

{\bf Remark.} The knot $K_m$ has bridge number 2, showing that two-bridge knots realize the full range of $\Xi_p$ when $p=3$. This answers a question posed in~\cite{kjuchukova2015classification}. It is not known whether the full range of $\Xi_p$ is realized by two-bridge knots when $p\neq 3$. It would be of interest to establish that it is ``sufficient" to consider two-bridge knots when constructing singular dihedral covers of four-manifolds since $p$-admissibility is particularly easy to detect for two-bridge singularities~\cite{kjorr2017admissible}.

\section{Proofs of {Theorems~\ref{bound}, ~\ref{example} and ~\ref{sum}}}
\label{proof}

\begin{proof}[Proof of Theorem~\ref{bound}]
	{\it (A)}
{Given a $p$-admissible knot $K$ with $p$-coloring $\rho$, we wish to prove the inequality:}

	$$\mathfrak{g}_{{p}}(K, \rho)\geq \dfrac{|\Xi_p(K,\rho)|-\rk H_1(M)}{p-1}-\dfrac{1}{2},$$
	where $M$ denotes the dihedral cover of $S^3$ branched along $K$ corresponding to $\rho$.	Throughout, all homology groups are with $\mathbb{Z}$ coefficients.

	For $K$ a $p$-admissible knot, by the definition of homotopy ribbon dihedral genus, there exists a topologically locally flat orientable homotopy ribbon surface $F'$ for $K$ such that the genus of $F'$ equals $\mathfrak{g}_p(K, \rho)$. (If $\rho$ does not extend over any locally flat, orientable, homotopy ribbon surface, $\mathfrak{g}_p(K, \rho)=\infty$ and the inequality is trivial.) Recall that, since $F'$ is orientable, $|\Xi_p(K,\rho)|=\sigma(X',\partial X')$; the righthand side denotes the Novikov signature of $X'$ as a manifold with boundary.  We will find an upper bound for $|\sigma(X',\partial X')|$ in terms of the Euler characteristic of $X=X'\cup c(\partial X')$.

	Let $\bar{\rho} :\pi_1(B^4-F')\rightarrow D_{p}$ be the homomorphism which extends the coloring $\rho:\pi_1(S^3-K)\rightarrow D_{p}$ and induces the cover $X'\to B^4$ branched over $F'$.  Let $\hat M$ be the unbranched irregular dihedral cover of $S^3-K$ corresponding to $\rho$, and $M$ the induced branched cover.  Denote by $F\subset S^4$ the singular surface which is the boundary union of $F'$ and the cone on $K$, so that $X$ is the dihedral cover of $S^4$ with branching set {$F$}.

	We will show that $X$ is simply-connected. Consider the commutative diagram below.  All maps in the diagram are either induced by inclusions or by covering maps. Clearly $p_*$ and $q_*$ are injective, as they are induced by covering maps, and $\mathsf{i}_{M*}$ and $\mathsf{i}_{X*}$ are surjective, as they are induced by inclusions of unbranched covering spaces into their branched counterparts.  The homomorphisms $\rho$ and $\bar{\rho}$ are surjective by definition.  Finally, since $F'$ is a homotopy-ribbon surface for $K$, $i_*$ is surjective.

	\begin{center}
	\begin{tikzcd}
			\pi_1(M)\arrow[r,"\mathsf{i}_*"]&\pi_1(X')\\
			\pi_1(\hat M)\arrow[r,"j_*"]\arrow[d,"p_*"]\arrow[u,"\mathsf{i}_{M*}"]&\pi_1(\hat X')\arrow[d,"q_*"]\arrow[u,"\mathsf{i}_{X*}"]\\
			\pi_1(S^3-K)\arrow[r,"i_*"]\arrow[d,"\rho"]&\pi_1(B^4-F')\arrow[ld,"\bar{\rho}"]\\
			D_{p}
			\end{tikzcd}
			\end{center}
	 We now show that $j_*$ is surjective as well.  Consider an element $\gamma\in\pi_1(\hat X')$. Since $i_*$ is surjective, there exists an element $\delta\in \pi_1(S^3-K)$ such that $i_*(\delta)=q_*(\gamma)$.  We have that $\bar{\rho}\circ q_*(\gamma)\in \mathbb{Z}/2\mathbb{Z}\subset D_p$, the reflection subgroup which determines the cover $\hat X'$ of $B^4-F'$. By commutativity of the lower triangle, $\rho(\delta)=\bar{\rho}\circ q_*(\gamma)\in  \mathbb{Z}/2$, so { $\delta\in \text{Im }p_*$.} Take $\tilde{\delta}\in \pi_1(\hat M)$ such that $p_*(\tilde{\delta})=\delta.$  Consider $q_*\circ j_*(\tilde{\delta})$, which by commutativity is equal to $i_*\circ p_*(\tilde{\delta})$.  Now $q_*\circ j_*(\tilde{\delta})=i_*(\delta)=q_*(\gamma)$. By injectivity of $q_*$, we have $j_*(\tilde{\delta})=\gamma$, so $j_*$ is indeed surjective.
	 
	Observe that, since $j_*$ and $\mathsf{i}_{X*}$ are both surjective, $\mathsf{i}_*:\pi_1(M)\rightarrow \pi_1(X')$ is surjective as well .

	By Seifert-van Kampen's theorem, we have $\pi_1(X)\simeq \pi_1(c(M))*_{\pi_1(M)}\pi_1(X')$. The cone $c(M)$ is contractible so $\pi_1(c(M))$ is trivial.  Hence
	$$\pi_1(c(M))*_{\pi_1(M)}\pi_1(X')\simeq\pi_1(X')/\text{im }\mathsf{i}_*(\pi_1(M)).$$
The quotient $\pi_1(X')/\text{im }\mathsf{i}_*$ is trivial since $\mathsf{i}_*$ is surjective.	 Hence $X$ is simply connected and, in particular, $H_1(X)=0$. We also know that $H_0(X)=1$ since $X$ is path-connected.

Next, consider $$\chi(X)=\rk H_4(X)-\rk H_3(X)+ \rk H_2(X)+1.$$  
Since $X=X'\cup c(\partial X')$, $$H_n(X)=H_n(X'/\partial{X'})=\tilde{H}_n(X',\partial X').$$   By Lefschetz Duality and Universal Coefficients, we see that
$$\rk H_3(X)=\rk H_3(X',\partial X')=\rk H^1( X')=\rk H_1(X'),$$
and
$$\rk H_4(X)=\rk H_4(X',\partial X')=\rk H^0(X')=\rk H_0(X')=1,$$
hence $$\chi(X)=2+\rk H_2(X)-\rk H_1(X').$$

  By \cite[Equation~1.1]{kjuchukova2018dihedral} we have 
  	\begin{equation}\label{chi}
  	\chi(X)=2p-\dfrac{p-1}{2}\chi({F})-\dfrac{p-1}{2}.	
  	\end{equation}

Recall that the Novikov signature $\sigma(X,\partial X')$ is the signature of the intersection form defined on the image of the map $\mathfrak{i}_\ast: H_2(X')\rightarrow H_2(X',\partial X')$ induced by inclusion. Thus,
\begin{equation} \label{inequality-chain}
	|\Xi_p(K,\rho)|=\sigma(X,\partial X')\leq \rk \text{im }\mathfrak{i}_\ast \leq \rk H_2(X',\partial X')=\rk H_2(X).
\end{equation}

The result follows by combining this inequality with the two formulas for $\chi(X)$ above. We substitute $$\chi(F)=2-2g(F)$$ into Equation~\ref{chi}, and $$\rk H_2(X)= \chi(X)-2+\rk H_1(X')$$ into~\ref{inequality-chain}, which gives
$$|\Xi_p(K,\rho)|\leq 2p-\dfrac{p-1}{2}(2-2g(F))-\dfrac{p-1}{2}-2+\rk H_1(X').$$

Simplifying, we obtain 

$$g(F)\geq \dfrac{|\Xi_p(K,\rho)|-\rk H_1(X')}{p-1}-\dfrac{1}{2}.$$
Finally, since the inclusion $\mathsf{i}:M\rightarrow X'$ induces a surjection on fundamental groups, we also know that $\rk H_1(M)\geq \rk H_1(X')$.  Hence 
$$g(F)\geq \dfrac{|\Xi_p(K,\rho)|-\rk H_1(M)}{p-1}-\dfrac{1}{2}.$$
\\

	{\it (B)} 		{Let $K$ be a $p$-admissible knot with respect to a coloring $\rho$ and let $F'\subset B^4$ be a homotopy ribbon, locally flat oriented surface with boundary $K$ such that $\rho$ extends over $F'$. Denote by $F\subset S^4$ the surface with singularity $K$ obtained as a boundary union of $F'$ and the cone on $K$, and denote by $X$ the dihedral cover of $S^4$ determined by the induced coloring of $F$.  We assume that $X$ is a definite manifold.  {In particular, $K$ is in fact strongly admissible with respect to the coloring $\rho$, and the corresponding branched cover $M$ of $S^3$ along $K$ is again $S^3$.} We then wish to show that the inequality~(\ref{inequality}) is sharp:
		$$\mathfrak{g}_{{p}}(K, \rho) = \dfrac{|\Xi_p(K,\rho)|}{p-1}-\dfrac{1}{2}.$$
		
		Precisely, we will show that the right hand-side of this equation equals the genus of $F'$. That is, $F'$ will be seen to realize the lower bound from (A) on the dihedral genus $\mathfrak{g}_{{p}}(K, \rho)$ of $K$.} 	
		
	Since $X$ is a definite manifold,  $\text{rk } H_2(X)=|\sigma(X)|$. {By the proof of (A), $X$ is simply connected;} by Poincar\'e duality we have $\chi(X)=2+\rk H_2(X)$ and hence $$|\Xi_p(K, \rho)|=|\sigma(X)|= \chi(X)-2.$$  On the other hand, denoting by $g(F)$ the genus of $F$, by  Equation~\ref{chi} we have $$\chi(X)=2p-\dfrac{p-1}{2}(2-2g(F))-\dfrac{p-1}{2}.$$ Putting these two equations together, we conclude that $$g(F)= \dfrac{|\Xi_p(K,\rho)|}{p-1}-\dfrac{1}{2}.$$ 
	By assumption, the coloring $\rho$ extends over $F'$, so  $$g(F')\geq\mathfrak{g}_{{p}}(K, \rho)\geq\dfrac{|\Xi_p(K,\rho)|}{p-1}-\dfrac{1}{2}.$$ Thus, $F'$ realizes the $p$-dihedral genus of $K$.

		{In the second part of the theorem, we assume in addition that $|\sigma(K)| = \frac{2|\Xi_p(K, \rho)|}{p-1} -1$, where $\sigma(K)$ is the signature of the knot $K$. We wish to show that the topological four-genus and the topological homotopy ribbon $p$-dihedral genus of $K$ both equal $\frac{|\sigma(K)|}{2}$.}

	The additional assumption here can be rewritten as $|\sigma(K)|=2\frak{g}_p(K,{\rho})$.  Murasugi's signature bound~\cite[Theorem 9.1]{murasugi1965certain} states that $g_4(K)\geq |\sigma(K)|/2$.
	Thus, we have $g_4(K)\geq \frak{g}_p(K,{\rho})$.  But $g_4(K)\leq \frak{g}_p(K)\leq {\frak{g}_p(K,\rho)}$ in general, so $g_4(K)= \frak{g}_p(K)$.
	\end{proof}

	\begin{proof}[Proof of {Theorem} \ref{example}]
Our aim is to show that the equalities		$$g_4(K_{m})=\mathfrak{g}_{{3}}(K_{m})=\dfrac{|\Xi_3(K_{m}, \rho_{m})|}{2} -\frac{1}{2}=m$$ 	hold for the 3-colored knots $K_m$ introduced in the previous section. In particular, it will follow that the generalized topological Slice-Ribbon Conjecture holds for these knots.  

By {\it (B)} of Theorem \ref{main}, it suffices to show that 
	 \begin{enumerate}
		\item Each $K_m$ is the boundary of a homotopy-ribbon surface $F_m'$ such that $\frak{g}_3(K)=g(F_m')$, and
		 \item The signature $\sigma(K_m)$ satisfies the equality $$|\sigma(K)|=\frac{2|\Xi_p(K, \rho)|}{p-1} -1$$ for $p=3$.
		 \end{enumerate}
		
	We first address (1).	Surfaces $F_m'$ realizing the lower bound on dihedral homotopy-ribbon genus for the knots $K_m$ are constructed in the proof of Proposition~\ref{covers}: we have shown $g(F_m')=m$ and $|\Xi_3(K_m)|=2m+1$, so
		
		$$\dfrac{|\Xi_p(K_m)|}{p-1}-\frac{1}{2}=m.$$
		 We note that, since the knots $K_m$ are two-bridge, each of them has a unique $3$-coloring (up to permuting the colors), so there is no distinction between $\frak{g}_p(K_m,\rho_m)$ and $\frak{g}_p(K_m)$. By construction, the surface $F_m'\subset B^4$ obtained by deleting a small neighborhood of the singularity $K_m$ is ribbon since $A_m\cup\overline{B}_m$ only bounds the cone on $K_m$, while the unlinks $B_m\cup\overline{C}_m$ and $C_m\cup\overline{A}_m$ bound standard unknotted disks in $B^4$.
		
We now address (2).	We will compute the signature $\sigma(K_m)$, and show it is equal to $2m=\frac{2|\Xi_p(K)|}{p-1} -1$.
	
	The signature of $K$ can be computed using the Goeritz matrix $G(K)$, the matrix of a quadratic form associated to a knot diagram via a checkerboard coloring, and hence a (not necessarily orientable) spanning surface;  this technique was introduced by Gordon and Litherland \cite{gordon1978signature}. The advantage of this technique is that the dimension of the Goeritz matrix associated to a projection of a knot may be much smaller than the dimension of the corresponding Seifert matrix; indeed, the dimension of $G(K_m)$ is 4 for all $m$.  
		
	Gordon and Litherland proved that the signature of a knot is equal the signature of the Goeritz matrix of a diagram of the knot plus a certain correction term: $\sigma(K)=\sigma(G(K))-\mu$. We start by computing the Goeritz matrix $G(K_m)$ and its signature.
	\begin{figure}[htbp]
	
	\includegraphics[width=4in]{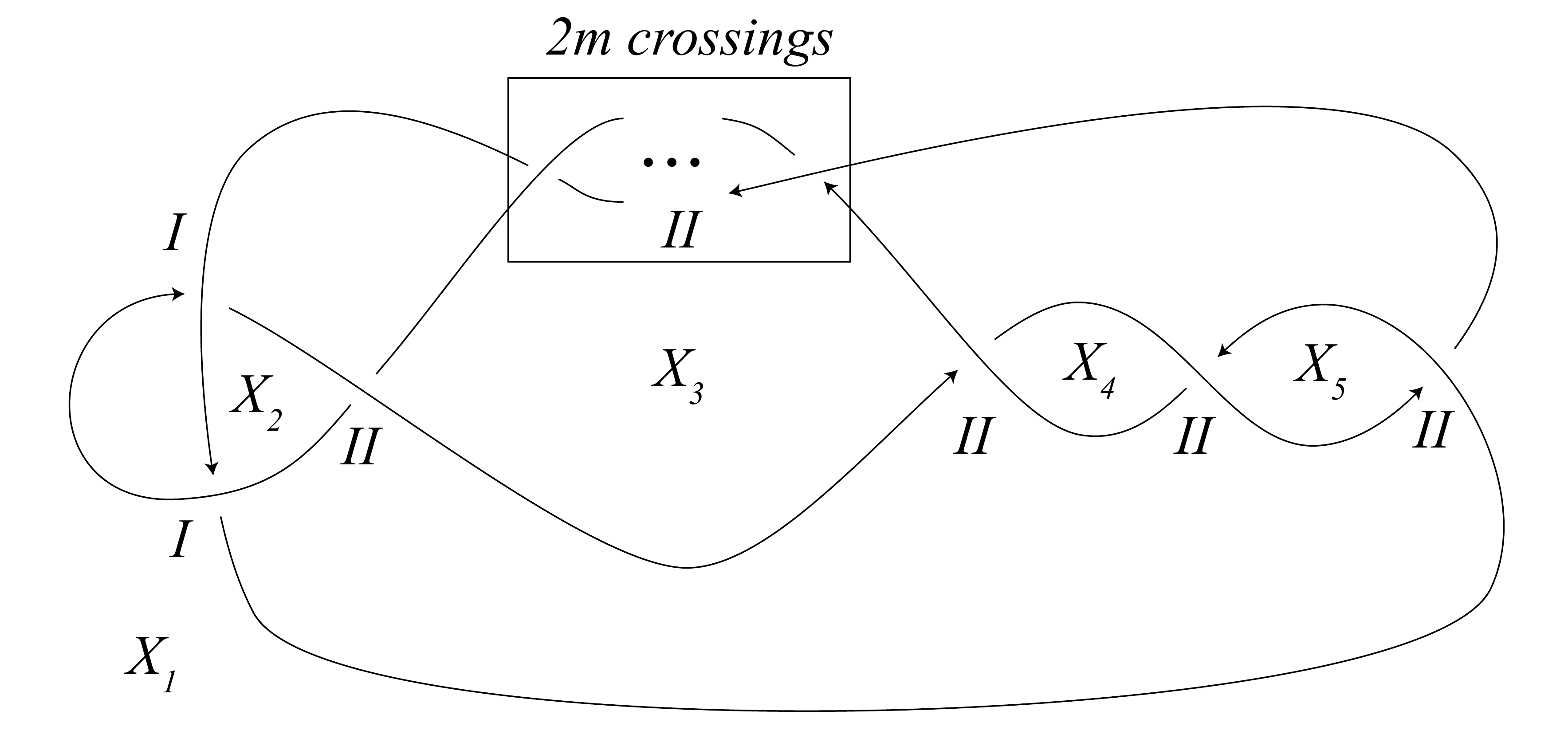}
	\caption{The ``white'' regions of a checkerboard coloring of $K_m$, labelled $X_1,X_2,\dots X_5$.}
	\label{singularityGoeritz.fig}
	\end{figure}
	
	One first computes the {\it unreduced} Goeritz matrix.  To do this, one chooses a checkerboard coloring of the knot diagram, and labels the ``white'' regions $X_1,\dots,X_k$.  Such a labelling for the $K_m$ is shown in Figure ~\ref{singularityGoeritz.fig}. The entries $g_{ij}$ of the unreduced Goeritz matrix are computed as follows:
	
	$$g_{ij}= \left\{ \begin{array}{ll} -\displaystyle{\sum} \eta(c)&i\neq j \text{ and }c \text{ a double point incident to regions } X_i \text{ and }X_j\\-\displaystyle{\sum}_{s\in \{1,\dots,k\}\setminus\{i\}} g_{is}& i=j \end{array}\right..$$
	
	The signs $\eta(c)$ are computed as in Figure ~\ref{Goeritzcrossings.fig}; shaded area corresponds to ``black'' regions of the checkerboard coloring.
 	\begin{figure}[htbp]
	
 	\includegraphics[width=4in]{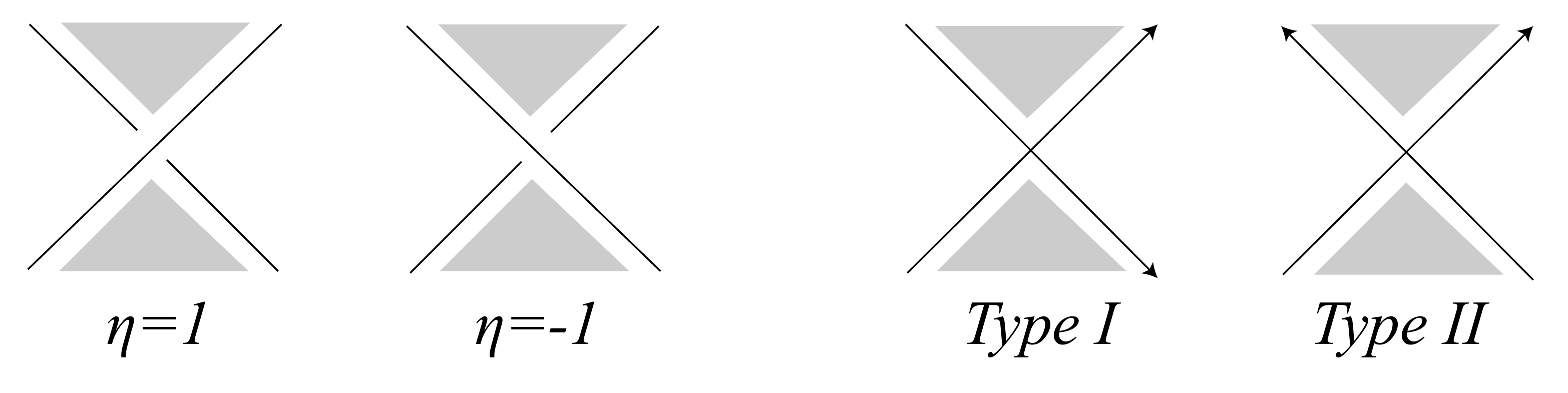}
 	\caption{Incidence numbers $\eta$ and Type $I$ and $II$ crossings.}
 	\label{Goeritzcrossings.fig}
 	\end{figure}

	The unreduced Goeritz matrix of $K_m$ is
$$G'(K_m)=\begin{pmatrix}-2m-3&-2&-2m&0&-1\\
-2&-3&-1&0&0\\
-2m &-1&-2m-2&-1&0\\
0&0&-1&-2&-1\\
-1&0&0&-1&-2
\end{pmatrix}.$$

The Goeritz matrix $G(K_m)$ is obtained by deleting the first row and column of $G'(K_m)$.  The characteristic polynomial of this matrix is
$$p_{G(K_m)}(\lambda)=(\lambda+3)(\lambda(\lambda+3)^2+2(\lambda+1)(\lambda+3)m+3).$$
Hence $\lambda=-3$ is an eigenvalue.  In addition, since $m\geq 0$, it is straightforward to verify that any root of the cubic factor must be negative (if $\lambda$ is nonnegative, the cubic, as written above, is a sum of three nonnegative terms).  Hence, $\sigma(G(K_m))=-4$. 

 The correction term $\mu(K)$ in Gordon and Litherland's formula for $\sigma(K)$ is computed as follows.  Each crossing $c$ of $K$ can be classified as type I or type II, as shown in Figure ~\ref{Goeritzcrossings.fig}.  Let $\mu(K)=\sum_c \eta(c)$ where the sum is taken over all type II crossings. 

The knot $K_m$ has $4+2m$ type II crossings, each of negative sign; see Figure ~\ref{singularityGoeritz.fig}. Hence $\sigma(K_m)=-4+(4+2m)=2m$. \end{proof}

\begin{proof}[Proof of Theorem~\ref{sum}] Given any non-negative integer $m$, our goal is to construct an infinite family of knots whose 3-dihedral and topological 4-genus are both equal to $m$. Let $K_m$ denote the knot given in Theorem~\ref{example} whose 3-dihedral and topological 4-genus equal $m$. We will prove that, given a nontrivial ribbon knot $\gamma$, the knot $K_m\#\gamma$ has the desired property. The theorem follows by taking repeated connect sums of $K_m$ with $\gamma$. 
	
Let $\gamma$ denote any ribbon knot and let $D\subset B^4$ be a ribbon disk with $\partial D=\gamma$. The knot $K_m\#\gamma$ has 3-dihedral genus and topological four-genus equal to $m$, {as we now show}.
It is clear that the smooth and topological four-genera of $K_m\#\gamma$ are both equal to $m$ since the knot is smoothly concordant to $K_m$. Next, remark that the given 3-coloring $\rho_m$ of $K_m$  induces a 3-coloring $\rho_\gamma$ of $K_m\#\gamma$ which restricts trivially to $\gamma$. Moreover, since $\rho_m$ extends over $F'_m$, $\rho_\gamma$ extends over the ribbon surface {$F_m'\natural D$, where $\natural$ denotes boundary connected sum}. 	Therefore, the ribbon 3-dihedral genus of $K_m\#\gamma$ is at most $m$. Since $g_4$ is a lower bound for the topological 3-dihedral genus, which in turn is a lower bound for the ribbon 3-dihedral genus, it follows that these genera are equal, as claimed. 
\end{proof}

Patricia Cahn\\
Smith College\\
{\it pcahn@smith.edu}

Alexandra Kjuchukova\\
Max Planck Institute for Mathematics--Bonn\\
{\it sashka@mpim-bonn.mpg.de}

\bibliographystyle{amsplain}
\bibliography{BrCovBib}

\end{document}